\documentclass[10pt,a4paper,reqno]{amsart}
\usepackage{amsmath,amsthm,amssymb}
\usepackage{mathrsfs,mathtools,nccmath,enumerate}
\usepackage{tikz-cd}
\usepackage{courier}
\usepackage[thinlines]{easytable}
\usepackage{amsfonts}
\usepackage{multicol}
\usepackage[mathscr]{euscript}
\usepackage{dirtytalk}
\usepackage{graphicx,caption,subcaption,float}
\usepackage{hyperref}
\usepackage{dirtytalk}
\usepackage{cleveref}
\usepackage{faktor}

\graphicspath{}
\newcommand{\interior}[1]{%
	{\kern0pt#1}^{\mathrm{o}}%
}
\usepackage[super]{nth}

\newtheorem{theorem}{Theorem}[section]
\newtheorem{lemma}[theorem]{Lemma}
\newtheorem{proposition}[theorem]{Proposition}
\newtheorem{corollary}[theorem]{Corollary}
\theoremstyle{definition}
\newtheorem{definition}{Definition}[section]
\newtheorem{conjecture}{Conjecture}[section]
\newtheorem{example}{Example}[section]
\theoremstyle{remark}

\DeclareMathOperator{\Tube}{Tube}
\newcommand{\R}{\mathbb R}

\newcommand{\Z}{\mathbb Z}

\newcommand{\LL}{\left\langle}
\newcommand{\RR}{\right\rangle}

\allowdisplaybreaks

\begin{document}	
\title{Non-triviality of Welded knots and ribbon torus-knots}
\author{Tumpa Mahato}
\address{Department of Mathematics, Indian Institute of Science Education and Research Pune, India}
\email{tumpa.mahato@students.iiserpune.ac.in}
\author{Rama Mishra}
\address{Department of Mathematics, Indian Institute of Science Education and Research Pune, India}
\email{r.mishra@iiserpune.ac.in}
\author{Sahil Joshi}
\email{2018maz0001@iitrpr.ac.in}

\begin{abstract}
    In this paper we study welded knots and their invariants. We focus on  generating examples of non-trivial knotted ribbon tori as the tube of  welded knots that are obtained from classical knot diagrams by welding some of the crossings. Non-triviality is shown by determining the fundamental group of the concerned welded knot.
    Sample examples under consideration are the standard diagrams of  the family of $(2, q)$ torus knots and the twist knots. Standard diagrams of knots from Rolfsen's tables 
    with $6$ crossings are also discussed which are not in the family of torus and twist knots. 
\end{abstract}

\maketitle

\section{Introduction}
Kauffman~\cite{Kau} developed virtual knot theory as a generalization of classical knot theory.
The idea of introducing a virtual crossing in link projections, which are $4$-regular planar graphs, is motivated by the drawing of a non-planar graph in $\R^2$ as well as the construction of a knot diagram from an arbitrary Gauss code (see~\cite{Kau}).
A \emph{virtual link} is represented by a virtual link diagram which is the union of a finite number of circles immersed in $\R^2$ (or $S^2$) having finitely many singularities, each of which is a transverse double point indicated by either a \emph{classical} or a \emph{virtual} crossing as shown in Figure~\ref{fig:crossings}.
Two virtual link diagrams are equivalent if one can be obtained from the other by a finite sequence of \emph{generalized Reidemeister moves} shown in Figures~\ref{fig:moves1} and~\ref{fig:moves2}.
Any two equivalent virtual link diagrams represent the same virtual link.
A virtual link consisting of one circle is called a virtual knot.
The local deformations $F_u$ and $F_o$ of virtual link diagrams shown in Figure~\ref{fig:forbidden} are called \emph{forbidden moves} for if they are allowed, every virtual knot becomes equivalent to a circle embedded in the plane $\R^2$, eventually rendering the theory to be trivial (see Kanenobu~\cite{Kan} or Nelson~\cite{Nel}).
However, if one allows $F_o$-move along with the generalized Reidemeister moves, the theory of \emph{welded links} is obtained which first appeared in Fenn, Rim\'{a}nyi, and Rourke~\cite{Fenn}.
This theory also generalizes classical knot theory, see for instance Rourke~\cite{Rou}.
Using Wirtinger presentation, the notion of knot group has been extended to virtual and welded knots, in a  combinatorial set up in~\cite{Kau}.
These are referred as the fundamental group of the virtual (welded) knot.
In case of classical knots, the knot group is an unknot detector.
It is known to be not true for virtual knots.
It is worth exploring if this is true for welded knots.\par

From a geometric point of view, virtual links are understood as the ambient isotopy classes of links in thickened oriented surfaces considered up to surface homeomorphisms and stablizations/destabilizations of handles disjoint from the links~\cite{Kup}.
But this interpretation does not work for  welded links. However, Satoh~\cite{Satoh1} found an interesting relationship between welded knots and certain embeddings of the torus in $4$-space (called the ribbon torus-knots), which essentially provides a geometric interpretation of welded knots, and is the focus of this paper.\par

Yajima~\cite{Yaj} showed that to each classical knot diagram, one can associate a ribbon torus-knot in $\R^4$ constructed by means of a \emph{tube map}, denoted by $K\mapsto\Tube(K)$.
This map was later extended in a natural way to the set of all virtual knot diagrams by Satoh~\cite{Satoh2}.
Conversely, Satoh proved that every ribbon torus-knot $T$ is associated with a virtual knot diagram $K$ whose image under the tube map, $\Tube(K)$, is ambient isotopic to $T$.
Such a pair $(T,K)$ is called a virtual knot presentation of $T$.
He also proved that for a given pair $(T,K)$  the fundamental group of $K$ is isomorphic to $\pi_1(\R^4\setminus T)$, see Theorem~\ref{thm:groupsTK}.\par

In~\cite[Proposition 3.3]{Satoh2}, Satoh proved that given any diagram of a welded knot, the associated ribbon torus-knot is uniquely determined up to ambient isotopy.
More precisely, we have Theorem~\ref{thm:Satoh}.
It was observed by Winter~\cite[Theorem 3.7]{Win} that the tube map fails to be injective.   Still many examples of nontrivial ribbon torus knots can be generated by simply finding
examples of welded knots which have their fundamental group not isomorphic to $\mathbb{Z}$. \par
 
Ribbon torus knots are a special class of genus $1$ surface knots~\cite{Satoh2}. 
In general, a surface-knot $F$ is a closed, connected, and orientable surface embedded locally flatly in $S^4$.
A surface knot $F$ is said to be \emph{topologically unknotted} if it bounds a locally flat, embedded handlebody in $S^4$, see~\cite{Conway}.
The unknotting conjecture states that $F$ is topologically unknotted if and only if $\pi_1(S^4\setminus F)$ is an infinite cyclic group.
This conjecture has been proved for $2$-knots (i.e. knotted spheres $S^2\hookrightarrow S^4$) in Freedman and Quinn~\cite{Freedman}, and is claimed for the genus $g\geq 1$ case by Hillman and Kawauchi~\cite{Hillman}.
Very recently, Conway and Powell~\cite{Conway} presented a new proof of the unknotting conjecture except for the genus $1$ (i.e. knotted tori $T\hookrightarrow S^4$) and $2$ cases. Despite the unknotting conjecture being open, a knotted torus $T$ will certainly be non-trivial if $\pi_1(S^4\setminus T)$ turns out to be not isomorphic to $\mathbb{Z}$.  It would amount to find a welded knot representation of $T$  whose fundamental group is not isomorphic to $\mathbb{Z}$.  \par

In a welded knot diagram, there are two types of crossings, the real crossings (which are either an over or  an under crossing) and the virtual crossings. Since we are dealing with welded knot theory, we call them welded crossings. We can replace each welded crossing by either an over or an under crossing and obtain a classical knot diagram.  Similarly, in a  classical knot diagram, we can replace few crossings with welded crossings. We  say that we have {\it welded}  those crossings. In nutshell, a welded knot diagram can be seen as if we have welded few crossings in a classical knot diagram. Thus starting with a classical knot diagram, we can generate many different welded knot diagrams and hence many different ribbon torus-knots by applying the tube map on them.  Those welded knot diagrams whose fundamental group is not isomorphic to $\mathbb{Z}$ are certainly non-trivial for which tube map will provide non-trivial ribbon torus knots. 

In this paper we have taken certain classical knot diagrams of some interesting families of knots such as {\it torus knots of type $(2, 2n+1)$ } and {\it twist knots} and study the nature of the welded knots obtained by welding some of the crossings of their standard diagrams. We find that the number as well as the positions of the crossings that have been welded both play important roles in determining the welded knot equivalence. In Rolfsen's table all the knots up to $5$ crossings are either torus or twist knots. First example of a knot outside these two families are in $6$ crossings. We have discussed them as well in this paper.

Since we are transforming classical crossings into welded ones, a related invariant of welded knots, the \emph{welded unknotting number}, is also of our interest.
The welded unknotting number of any welded knot is the minimum number of classical crossings required to be changed to welded crossings to obtain the unknot.
We also discuss its relationship with the unknotting number \cite{Satoh1} and the warping degree of a welded knot~\cite{Li2}.\par

This paper is organised as follows: In Section 2 we have included the basic definitions. It has two subsections. Section 2.1 discusses the basics of virtual and welded knot theory 
and Section 2.2 includes the definition of ribbon torus-knots and its connection with welded knots.  In Section 3, we talk about the {\it  warping degree}  (in Section 3.1) of welded knots
and in Section 3.2, we discuss the {\it welded unknotting number} and include some of our results (inequalities) on welded unknotting number and its relation with the {\it unknotting number}. Section 4 contains our main results. Here we take the standard diagrams of some families of knots, weld some of their crossings and compute the fundamental group of the resulting welded knots.
For the torus and twist knot family we have restricted our study to welding at most $2$ crossings for the purpose of computational simplicity. At the end of Section 4 we have discussed the welding of all the $6$ crossing knots in detail. In Section 5, we mention some remarks and list some future problems.

\section{Prerequisites}
\subsection{Virtual and welded knots}
\begin{definition}
    A \emph{virtual link diagram $K$} is the image of a generic immersion $f:\sqcup_{i=1}^n S^{1}\to\R^{2}$ such that each double point of $K$ is equipped with the structure of a crossing of one of the two types, namely classical or virtual, shown in Figure~\ref{fig:crossings}.
\end{definition}
\begin{figure}[h]
    \centering
    \begin{subfigure}{0.49\textwidth}
        \centering\includegraphics{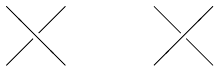}
        \caption{Classical crossings.}
    \end{subfigure}
    \hfill
    \begin{subfigure}{0.49\textwidth}
        \centering\includegraphics{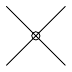}
        \caption{Virtual crossing.}
    \end{subfigure}
    \caption{Two types of crossings.}
    \label{fig:crossings}
\end{figure}
Two virtual link diagrams are said to be equivalent if one can be transformed into the other by a finite sequence of classical Reidemeister moves $R_{1},R_{2},R_{3}$ (see Figure~\ref{fig:moves1}) and virtual Reidemeister moves $vR_{1},vR_{2},vR_{3},vR_{4}$ (see Figure~\ref{fig:moves2}).
This notion defines an equivalence relation on the set of virtual link diagrams.
\begin{definition}
    A \emph{virtual link} is the equivalence class of virtual link diagrams under the equivalence relation generated by classical and virtual Reidemeister moves. 
\end{definition}
Each virtual link of one component is called a virtual knot.
In this paper, we are only concerned about virtual knots.
There are two other local moves --- the over-forbidden move $F_o$ and the under-forbidden move $F_u$ (see Figure~\ref{fig:forbidden}) --- which are not allowed on virtual knots diagrams due to the following result independently proved by Kanenobu~\cite{Kan} and Nelson~\cite{Nel}.
\begin{theorem}[\cite{Kan}]
    For any virtual knot $K$, there exists a finite sequence of generalized Reidemeister moves and $F_o$- and $F_u$-moves that takes $K$ to a trivial knot.
\end{theorem}
\begin{figure}[h]
    \centering
    \includegraphics{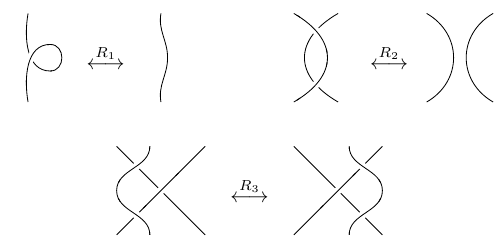}
    \caption{Reidemeister moves.}
    \label{fig:moves1}
\end{figure}
\begin{figure}[h]
    \centering
    \includegraphics{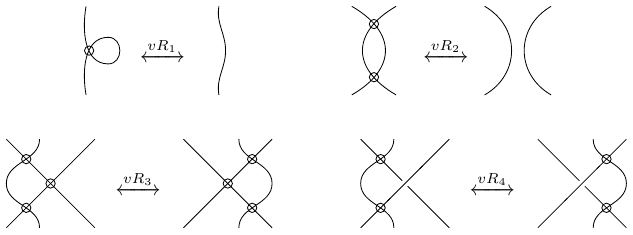}
    \caption{Virtual Reidemeister moves. }
    \label{fig:moves2}
\end{figure}
\begin{figure}[h]
    \includegraphics[scale=1]{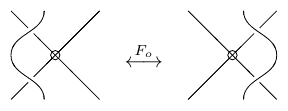}
    \qquad
    \includegraphics[scale=1]{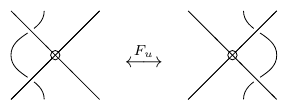}
    \caption{Forbidden moves.}
    \label{fig:forbidden}
\end{figure}
Welded knot theory is a quotient of virtual knot theory obtained by allowing the $F_o$-move.
\begin{definition}
    A welded knot is the equivalence class of all those virtual knot diagrams which are related by a finite sequence of classical and virtual Reidemeister moves, and the over forbidden move.
\end{definition}
The fundamental group of a virtual knot $K$, denoted by $G(K)$, is defined by generalizing the Wirtinger presentation of the groups of classical knots in a combinatorial manner.
The part of a virtual knot diagram between any two consecutive undercrossings is called an arc of the diagram.
The set of labeled arcs of $K$ forms a generating set of $G(K)$ and for each classical crossing, the relation between the labels are as mentioned in Figure~\ref{fig:relation} holds.
There are no relations for the virtual crossings.
\begin{figure}[h]
    \centering
    \includegraphics{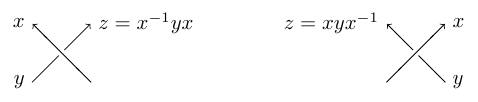}
    \caption{The Wirtinger labeling of a crossing.}
    \label{fig:relation}
\end{figure}

\subsection{Ribbon torus-knots and their association with virtual knots}
\begin{definition}
    Let $M_{1}$ and $M_{2}$ be two $3$-dimensional manifolds immersed in $S^4$.
    Then their intersection $D = M_{1} \cap M_{2}$ is called a ribbon singularity if $D$ is homeomorphic to $D^{2}$ such that $D \subset \partial M_{1}$ and $ \partial D \subset \partial M_{2} $.
\end{definition}
\begin{definition}
    A ribbon torus is an immersed solid torus $D^{2} \times S^{1} \subset S^{4}$ whose singular set consists of a finite number of ribbon singularities.
\end{definition}
\begin{definition}
    A ribbon torus-knot is an embedded torus $S^{1} \times S^{1} \subset S^{4}$ which bounds a ribbon torus.
\end{definition}

To a virtual knot diagram we associate a ribbon torus knot diagram as follows: At a classical crossing, the tube corresponds to the under arc will go through the tube corresponds to the over arc; and at a virtual crossing, the corresponding tubes are independent of the over-under information (see Figure~\ref{fig:tubemap}).
This correspondence is called the tube map.
\begin{figure}[h]
    \centering
    \includegraphics[width=0.4\textwidth]{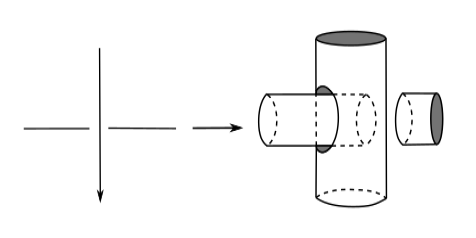}
    \qquad
    \includegraphics[width=0.4\textwidth]{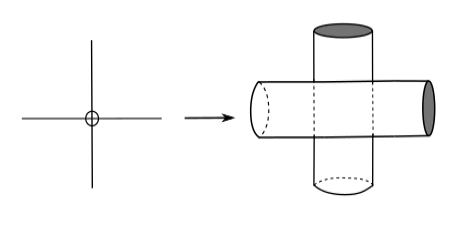}
    \caption{Tube map.}
    \label{fig:tubemap}
\end{figure}

Satoh proved the following results.
\begin{theorem}[\cite{Satoh2}]
    Any ribbon-torus knot $T$ is associated with some virtual knot diagram $K$; $T\cong\Tube(K)$.
\end{theorem}
\begin{theorem}[\cite{Satoh2}]\label{thm:groupsTK}
    Let $(T,K)$ be a virtual knot presentation of a ribbon torus-knot $T$ and let $G(K)$ be the fundamental group of $K$.
    Then $\pi_1(\R^4\setminus T) \cong G(K)$.
\end{theorem}
\begin{theorem}[\cite{Satoh2}]\label{thm:Satoh}
    If $K$ and $K'$ are equivalent welded knots, then $\Tube(K)$ and $\Tube(K')$ are equivalent as ribbon torus-knots.
\end{theorem}
	
It is easy to see that the fundamental group of the welded unknot $K$ is isomorphic to $\mathbb{Z}$ and hence from  Theorem 2.3, 
$\pi_1(\R^4\setminus T(K)) \cong \mathbb{Z}$. If the unknotting conjecture \cite{Conway} is true, it will imply that $T(K)$ is trivial. However, if  $\pi_1(\R^4\setminus T) \ncong \mathbb{Z}$	then $T(K)$ is non-trivial. Thus for obtaining various examples of non-trivial knotted ribbon tori, we can realise them as tube of such welded knots $K$ whose fundamental group $G(K)\ncong \Z$.
In Section 4, we study many such examples where our welded knot diagrams are obtained from a classical knot diagram after welding some of its crossings.  We compute the fundamental group in each case and check if that is not isomorphic to $\mathbb{Z}$. This gives us a machinery to obtain non-trivial knotted ribbon tori.  In case the fundamental group of a welded knot turns out to be $\mathbb{Z}$, we cannot conclude that it is trivial.  In such situations some other strategies are used. In few examples one can use the generalised Reidemeister moves  to show that it is indeed an unknot. In general, it is difficult to conclude.

For a classical knot diagram which represents a non-trivial knot, welding a certain number of crossings  will always yield an unknot \cite{Kau}.  In Section 3, we discuss some numerical invariants of welded knots that may be useful when it comes to concluding the non-triviality of a welded knot.

\section{Some numerical invariants of welded knots}

\subsection{Warping degree}

Let $D$ be an oriented welded knot diagram of a welded knot $K$.
Let $a$ be a point on $D$ which is not a crossing point.
We will call it a base point of $D$.
\begin{definition}
    The warping degree of $D$ with the base point $a$, denoted by $d(D_{a})$, is the number of crossing points of $D$, we meet as a under-crossing at the first encounter while we go along the orientation of $D$.\\
    The warping degree of $D$, denoted by $d(D)$, is the minimal warping degree for all base points of $D$.
    The warping degree of $K$ is defined by
    \[d(K) \coloneqq \min_{D}\{d(D),d(-D) \mid \text{$D$ is a diagram of $K$}\}.\]
\end{definition}
\begin{figure}[h]
    \centering
    \includegraphics[width=0.8\textwidth]{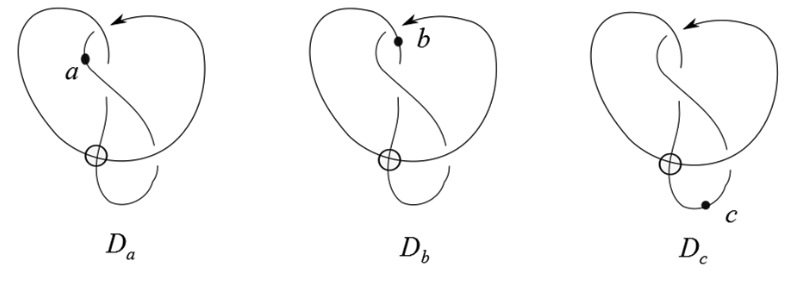}
    \caption{The welded figure-eight knot diagram $D$. }
    \label{fig:figure8}
\end{figure}
\begin{figure}[h]
    \includegraphics[width=0.8\textwidth]{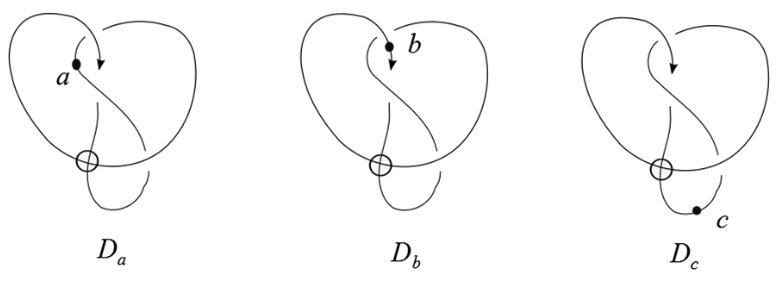}
    \caption{The welded figure eight knot diagram $-D$.}
    \label{fig:figure82}
\end{figure}
\begin{example}
    Let $D$ be a diagram of a welded figure-eight knot $K$ with one welded crossing as in Figure \ref{fig:figure8}.
    Here, we have $d(D_{b} ) = 0, d(D_{a} ) = 1, d(D_{c} ) = 1$. Hence, we have $d(D) = 0$.
    For the inverse $-D$ of $D$ (see Figure \ref{fig:figure82}), we have $d(-D_{a} )= 2, d(-D_{b} )=3,d(-D_{c} )=2$.
    Hence, we have $d(-D)=2$ and $d(K)=0$.
\end{example}
\begin{lemma}[\cite{Li}]
    The unknotting number of a welded knot is less or equal to its warping degree i.e
    \[u(K) \leq d(K).\]
\end{lemma}
\begin{corollary}[\cite{Li}]\label{cor1}
    A welded knot is trivial if and only if $d(K)=0$.
\end{corollary}	
\begin{lemma}
    Let $x$ be a classical crossing of a welded knot diagram $D$.
    We divide $D$ into two o closed paths by cutting  $D$ at $c$.
    Suppose that one of the obtained paths contains no under-crossing except $x$.
    Let $E$ be the welded knot diagram obtained from  $D$ by replacing $x$  with a welded crossing.
    Then  $D$ is related to $E$ by a finite sequence of $C1, V1-V4$, and $W$.
\end{lemma}
\begin{definition}
    A welded knot diagram  $D$ is \textit{descending} if there is a base point and an orientation of  $D$ such that, walking along  $D$ from the base point  with respect to the orientation, we meet the over-crossing for the first time and the under-crossing for the second time at every classical crossing.
\end{definition}
\begin{proposition}\cite{Satoh1}\label{prop1}
    Any descending diagram is related to the trivial diagram by a finite sequence of $R_1, vR_1$--\,$vR_4$,  and $F_o$.
\end{proposition} 
\subsection{Welded unknotting number}

\begin{definition}
    For any welded knot diagram $D$, its \emph{welded unknotting number}, denoted by $u_{w}(D)$, is defined as the minimum number of classical crossings that need to be changed to welded crossings such that the resultant diagram becomes a diagram of the unknot.\par
    The \emph{welded unknotting number of} $K$, denoted by $u_{w}(K)$, is defined by
    \[ u_{w}(K) \coloneqq \min_{D}\{u_{w}(D) \mid D \text{ is a diagram of } K\}. \]
\end{definition}
\begin{proposition}
    For every welded knot $K$, the welded unknotting number does not exceed the warping degree, i.e.,
    \[ u_{w}(K) \leq d(K). \]
\end{proposition}
\begin{proof}
    Let $d(K) = r$.
    Suppose $D_a$ is an oriented welded knot diagram of $K$ with a base point $a$ that realizes the warping degree of $K$, i.e. $d(D_a) = r$.
    It implies that if we follow $D_a$ along its orientation, then we first meet the under-crossings at exactly $r$ number of crossing points of $D_a$.
    By transforming these $r$ crossings into welded crossings, we obtain a descending welded knot diagram, which is in fact a trivial welded knot diagram, as proved by Satoh~\cite[Proposition 2.2]{Satoh1}.
    Thus $u_{w}(K) \leq r = d(K)$.
\end{proof}
\begin{proposition}
    For every welded knot $K$, the welded unknotting number does not exceed twice of the unknotting number, i.e.,
    $$ u_{w}(K) \leq 2 u(K). $$
\end{proposition}
\begin{proof}
    Let $u(K) = n$.
    Then there exists a welded knot diagram $D$ representing $K$ with crossings $c_1, c_2, \ldots, c_n$ on $D$ such that by switching all $c_i$'s, we obtain a trivial welded knot diagram, say $\widetilde{D}$.
    Apply an $\mathrm{RII}$-move locally in a neighborhood of $c_i$ for each $i=1,2,\ldots,n$ to obtain another diagram $D'$ of $K$, as shown in Figure~\ref{fig:cc_weld}.
    Observe that if we change both $c_i$ and its adjacent crossing to welded crossings, we obtain a diagram that is welded equivalent to $\widetilde{D}$.
    This implies $u_w(K) \leq u_w(D') \leq 2n = 2 u(K)$.
    \begin{figure}[h]
	\centering 
	\includegraphics[scale=1]{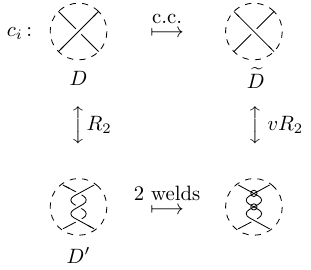}
	\caption{Realizing a crossing change by two weldings.}
	\label{fig:cc_weld}
    \end{figure}
\end{proof}
Let $T_n$ denotes the twist knot with $n$ half-twists having $n+2$ crossings.
\begin{proposition}
$T_n$ has welded unknotting number $1$ for all $n$.  
\end{proposition}
\begin{proof}
    Note that $T_1$ is the trefoil and $T_2$ is the figure eight knot.
    We first prove that welding one single crossing at any position in $T_n$ for $n=1$ and $2$ yields a welded knot equivalent to the  unknot.
    In Figure \ref{fig:trefoil} and \ref{fig:figure83} we can see that the warping degree for trefoil and figure eight knot with one welding turns out to be zero.  Therefore by Corollary~\ref{cor1}  both of them are unknots.
    \begin{figure}[h]
        \centering
        \includegraphics[scale=0.8]{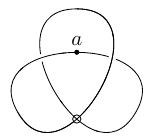}
	\caption{$d(D_{a})=0$ for the diagram of trefoil with one weld.}
	\label{fig:trefoil}
    \end{figure}
    \begin{figure}[h]
        \centering
        \includegraphics[width=0.18\textwidth]{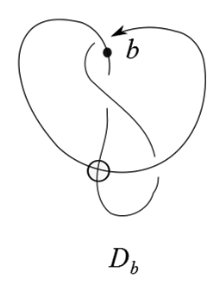}
        \caption{$d(D_{b})=0$ for the diagram of the figure-eight knot with one weld.}
        \label{fig:figure83}
    \end{figure}
    \begin{figure}[h]
        \centering
	\includegraphics[width=0.8\textwidth]{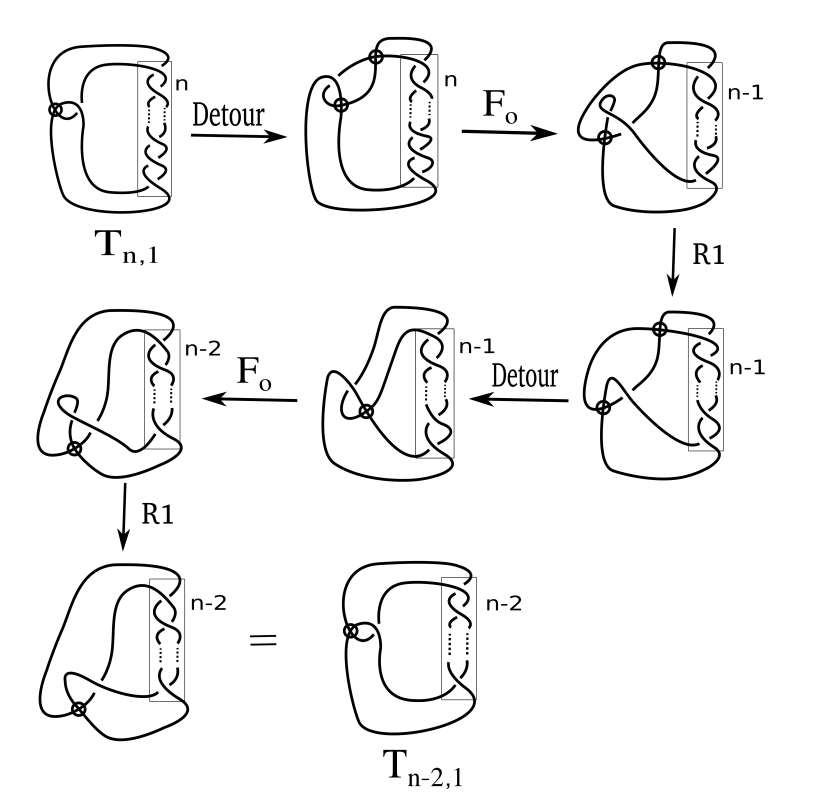}
	\caption{Triviality of the twist knot with one weld.}
	\label{fig:unknot}
    \end{figure}
    Let the standard diagram of twist knot be also denoted by $T_{n}$. Denote the welded diagram  obtained by   welding the left most crossing in $T_{n}$ by
    $T_{n, 1}$ as shown in Figure~\ref{fig:unknot}). In the Figure \ref{fig:unknot}, we have shown that $T_{n,1}$ is equivalent to $T_{n-2, 1}$ using detour move, forbidden move $1 $ and first Reidemeister move. Applying this process multiple times we will get $T_{2, 1}$ (if $n$ is odd) or $T_{1, 1}$ (if $n$ is even) which are shown to be unknot.
    \end{proof}
	
\begin{proposition}
    Let $K(2, 2n+1)$ denote the torus knot of type $(2, 2n+1)$. Then
    $$u_{w}(K(2,2n+1)) \leq u(K(2,2n+1))= n .$$
\end{proposition}
\begin{proof}
    We denote the standard oriented diagram of $K(2,2n+1)$ by $K(2,2n+1)$ itself. See Figure \ref{fig:Dk}.
    The classical unknotting number of $K(2,2n+1)$ is $\frac{(2-1)(2n+1-1)}{2}= n$, as proved by Kronheimer and Mrowka~\cite{KM}.
    \begin{figure}[h]
	\includegraphics[scale=0.75]{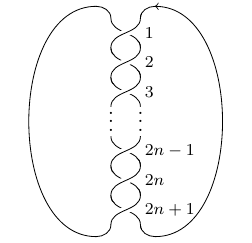}
	\caption{The standard diagram of the torus knot $K(2,2n+1)$.}
	\label{fig:Dk}
    \end{figure}
    The Gauss code for $K(2,2n+1)$ in Figure \ref{fig:Dk} is 
    \[ O1\;U2\;O3\;U4 \cdots \; O(2n-1)\;U2n\;O(2n+1)\;U1\;O2\;U3\;O4\cdots O2n\;U(2n+1)\]
    Now, if we take any base point on $K(2,2n+1)$ and go along the orientation then we will meet $n$  number of under-crossings at the first encounter with that crossing. And these under-crossings will be on alternate positions.\\
    Now, we will weld $n$ number of alternate crossings, say at $\nth{2},\nth{4},\cdots,2n^{th}$ crossing from the top as shown in Figure \ref{fig:torusalt}. We denote this oriented diagram by $D_{n}'$.
    \begin{figure}[h]
        \centering
	\includegraphics[scale=0.8]{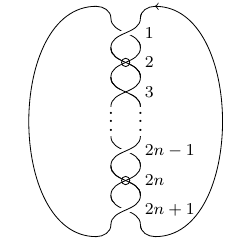}
	\caption{$K(2,2n+1)$ with $n$ welding at alternate crossings.}
	\label{fig:torusalt}
    \end{figure}
    Now, the Gauss code of this welded oriented diagram is as follows
    \[O1\,O3\,\cdots\,O(2n-1)\,O(2n+1)\,U1\,U3\,\cdots\,U(2n-1)\,U(2n+1)\]
    So, we get a descending diagram and by Proposition \ref{prop1}, it will be equivalent to an unknot.
    Another way to prove this is using warping degree. As we can see in the Figure \ref{fig:torusalt},
    the warping degree of $D_{n}'$ for base point a is $0$. Therefore $d(K(2,2n+1))=d(D_{n}')=0$. By Corollary~\ref{cor1}, $D_{n}'$ is an unknot.
    Both the arguments imply that the welded unknotting  number $u_{w}(K(2,2n+1)$ does not exceed $u(K(2,2n+1)$.
\end{proof}

\section{Fundamental group of welded Knots  obtained from the classical knot  diagrams}

In this section we consider the standard diagrams of two important families of knots, namely the torus knots of type $(2,2n+1)$ denoted by $K(2,2n+1)$ and the twist knots with $n$ half twists denoted by $T_n$. These knots are non-trivial and therefore have their knot groups  not isomorphic to $\mathbb{Z}$. Thus by Satoh's  result their tubes will be non-trivial knotted ribbon tori. We would like to see the impact on  fundamental groups when we weld some crossings in these diagrams.

\subsection{ Welded knots arising from  $K(2,2n+1)$}

\subsubsection{$K(2,2n+1)$ with one welding:}

Let $K_{n,1}$ be the welded knot represented by the diagram  shown in Figure~\ref{fig:torus1}, which is obtained from $K(2,2n+1)$ by welding one crossing.  Let $G(K_{n,1})$ denotes the fundamental group of $K_{n,1}$.
 \begin{figure}
	\centering
	\includegraphics[scale=0.75]{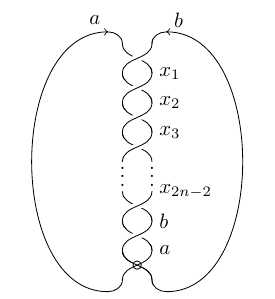}
	\caption{$K(2,2n+1)$ with one weld.}
	\label{fig:torus1}
    \end{figure}

\begin{proposition}
  
 $G(K_{n,1}) \cong \Z$.
   
\end{proposition}

\begin{proof}
    The Wirtinger relations at the crossings of the diagram $D$ whose arcs are labeled as shown in Figure \ref{fig:torus1} are given by
    \begin{align*}
        x_k & = 
        \begin{cases}
            (ba)^{\frac{k+1}{2}} a (ba)^{-{\frac{k+1}{2}}}  & \text{if $k$ is odd, $k\leq 2n-3$}\\
            (ba)^{\frac{k}{2}} b (ba)^{-{\frac{k}{2}}}      & \text{if $k$ is even, $k\leq 2n-2$}
        \end{cases}\\
        b & = x_{2n-2} x_{2n-3} x_{2n-2}^{-1}\\
        & = (ba)^n a (ba)^{-n}\\
        a & = b x_{2n-2} b^{-1}\\
        & = b (ba)^{n-1} b (ba)^{-(n-1)} b^{-1}.
    \end{align*}
    Then a Wirtinger presentation of $G(K_{n,1})$ is given by
    \begin{equation*}
        G(K_{n,1}) = \LL a,b: b(ba)^n = (ba)^n a,\ ab(ba)^{n-1} = b(ba)^{n-1}b \RR.
    \end{equation*}
    Observe that $ab(ba)^{n-1} = b(ba)^{n-1}b = b(ba)^n (ba)^{-1}b = (ba)^n a (ba)^{-1}b = (ba)^n$ implies $ab = ba$, which is a consequence of the relations in $G(K_{n,1})$.
    Further, $(ba)^n a = b(ba)^n = b(ab)^n = (ba)^n b$ implies $a=b$.
    Thus $G(K_{n,1}) \cong \LL a: \RR$, which is isomorphic to $\Z$.
\end{proof}

\noindent{Remark:}  Thus in the above situation we cannot predict if the welded knot or the tube associated to it is trivial   or non-trivial.

\subsubsection{$K(2,2n+1)$ with two weldings:}

We start with the standard diagram $K(2,2n+1)$. Welding two crossings divide the diagram in two parts with $m_{1}$ and $m_{2}$ number of classical crossings as shown in the Figure \ref{fig:randomgap}.
Then
\[m_{1}+m_{2}=(2n+1)-2=2n-1 (\text{odd}).\]
We denote this welded knot diagram as $K_{n,2,m_{1}}$ (Figure~\ref{fig:randomgap}).
Now, because of the cyclic nature of the standard diagram of $K(2,2n+1)$ the following pairs of $\{m_{1},m_{2}\}$ represent the same welded diagram:
\begin{align*}
    \{i,j\} & \; \text{and} \; \{j,i\} \quad \text{for } i,j=0,1,\cdots, 2n-1.
\end{align*}
So, we only consider the values $m_{1}=0,1,2,\cdots,(n-1)$ for $K_{n,2,m_{1}}$. Now,$K_{n,2,0}$ is equivalent to the classical diagram of $K(2,2n-1)$ under virtual Reidemeister move $vR_{2}$ (Figure \ref{zerogap}), which has non-$\mathbb{Z}$ knot group.
Let $G(K_{n,2,m_{1}})$ denote the (combinatorial) knot group of the welded knot $K_{n,2,m_{1}}$ where $0 \leq m_{1}\leq (n-1)$.
 \begin{figure}[h]
    \centering
    \includegraphics[scale=0.75]{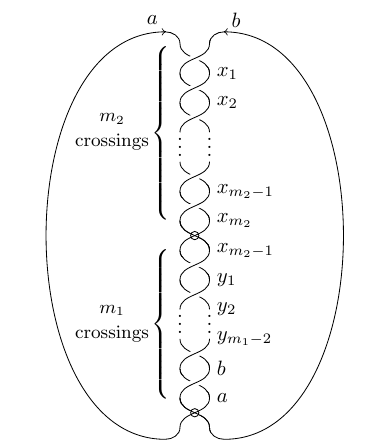}
    \caption{$K(2,2n+1)$ with two welds and arbitrary gap.}
    \label{fig:randomgap}
\end{figure}
\begin{proposition}
    \begin{align*}
        G(K_{n,2,m_{1}})  & \cong \mathbb{Z} \quad \text{when} \quad m_{1}=n-1\\
    G(K_{n,2,m_{1}})  &\ncong \mathbb{Z} \quad \text{when} \quad 1 \leq m_{1} < n-1.
    \end{align*}
\end{proposition}
   
\begin{figure}[h]
    \centering 
    \includegraphics[scale=0.8]{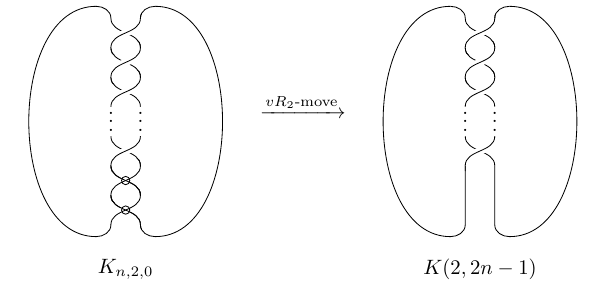}
    \caption{$K(2,2n+1)$ with two welds and zero gap.}
    \label{zerogap}
\end{figure}
\begin{proof}
To find a presentation of the knot group $G(K_{n,2,m_{1}})$, let the Wirtinger generators be $x_{1},x_{2},\cdots,x_{m_2-1}$ and $y_{1},y_{2},\cdots,y_{m_{1}-2},b,a$ (Figure \ref{fig:randomgap}).
Then the Wirtinger relations are
\begin{align*}
    x_{1} &= bab^{-1}\\
    x_{2} &= baba^{-1}b^{-1}.
\end{align*}
By induction, for each $i=1,2,\ldots,m_{2}$
\begin{align*}
    x_{i} &=
    \begin{cases}
        \underbrace{ba\cdots ba}_{i} b \underbrace{a^{-1}b^{-1}\cdots a^{-1}b^{-1}}_{i} & \text{when}\;i \;\text{is even} \\
        \underbrace{ba\cdots ba}_{i-1} bab^{-1} \underbrace{a^{-1}b^{-1}\cdots a^{-1}b^{-1}}_{i-1} &\text{when}\;i \;\text{is odd} \\
    \end{cases}
\end{align*}
Therefore,
$ x_{m_{2}}=
\begin{cases}
    \underbrace{ba\cdots ba}_{m_{2}} b \underbrace{a^{-1}b^{-1}\cdots a^{-1}b^{-1}}_{m_{2}}&  \text{when} \;m_{2} \;\text{is even} \\
    \underbrace{ba\cdots ba}_{m_{2}-1} bab^{-1} \underbrace{a^{-1}b^{-1}\cdots a^{-1}b^{-1}}_{m_{2}-1}  & \text{when} \;m_{2} \; \text{is odd}
\end{cases}$\\
$\bullet$ $m_{2}$ is even and $m_{1}$ is odd:
\begin{align*}
    y_{1}&= x_{m_{2}-1}x_{m_{2}}x_{m_{2}-1}^{-1}\\
    &=  \underbrace{ba\cdots ba}_{m_{2}}aba^{-1} \underbrace{a^{-1}b^{-1}\cdots a^{-1}b^{-1}}_{m_{2}}\\
    y_{2} &=  \underbrace{ba\cdots ba}_{m_{2}}abab^{-1}a^{-1}\underbrace{a^{-1}b^{-1}\cdots a^{-1}b^{-1}}_{m_{2}}\\
    \vdots & \qquad \quad  \vdots 
\end{align*} 
{\small \begin{align*}
y_{i}&=
    \begin{cases}
        \underbrace{ba\cdots ba}_{m_{2}}\underbrace{ab\cdots ab}_{i}\,a\,\underbrace{b^{-1}a^{-1}\cdots b^{-1}a^{-1}}_{i}\underbrace{a^{-1}b^{-1}\cdots a^{-1}b^{-1}}_{m_{2}} & \text{when}\;i \;\text{is even} \\
        \underbrace{ba\cdots ba}_{m_{2}}\,\underbrace{ab\cdots ab}_{i-1}\,aba^{-1}\,\underbrace{b^{-1}a^{-1}\cdots b^{-1}a^{-1}}_{i-1}\,\underbrace{a^{-1}b^{-1}\cdots a^{-1}b^{-1}}_{m_{2}} & \text{when}\;i \;\text{is odd}
    \end{cases}
\end{align*}}
Then
\begin{align*}
    b &= y_{m_{1}-2}\,y_{m_{1}-3}\,y_{m_{1}-2}^{-1} \\
    &=  \underbrace{ba\cdots ba}_{m_{2}}\; \underbrace{ab\cdots ab}_{m_{1}-1}\;a\;\underbrace{b^{-1}a^{-1}\cdots b^{-1}a^{-1}}_{m_{1}-1}\;\underbrace{a^{-1}b^{-1}\cdots a^{-1}b^{-1}}_{m_{2}}
    \end{align*} 
\begin{align*}
     \implies  \underbrace{ba\cdots ba}_{m_{2}}\, \underbrace{ab\cdots ab}_{m_{1}-1} &= \underbrace{ab\cdots ab}_{m_{2}-2}\,a\, \underbrace{ab\cdots ab}_{m_{1}-1}\,a 
\end{align*}
\begin{equation*}
    (ba)^{\frac{m_{2}}{2}}\,(ab)^{\frac{m_{1}-1}{2}}= (ab)^{\frac{m_{2}-2}{2}}\,a\, (ab)^{\frac{m_{1}-1}{2}}\,a \tag{R1}
\end{equation*}
and
{\small \begin{align*}
    a&= b\,y_{m_{1}-2}\,b^{-1}\\
    &= b\,  \underbrace{ba\cdots ba}_{m_{2}}\, \underbrace{ab\cdots aba}_{m_{1}-2}\,b\, \underbrace{a^{-1}b^{-1}\cdots a^{-1}b^{-1}a^{-1}}_{m_{1}-2}\,\underbrace{a^{-1}b^{-1}\cdots a^{-1}b^{-1}}_{m_{2}}\,b^{-1}
\end{align*}}
\begin{equation*}
     \implies ab\,  \underbrace{ba\cdots ba}_{m_{2}}\,a\, \underbrace{ba\cdots ba}_{m_{1}-3} = b\, \underbrace{ba\cdots ba}_{m_{2}}\, \underbrace{ab\cdots ab}_{m_{1}-1}
\end{equation*}
\begin{equation*}
    ab\, (ba)^{\frac{m_{2}}{2}}\,a\,(ba)^{\frac{m_{1}-3}{2}}= b\,(ba)^{\frac{m_{2}}{2}}\,(ab)^{\frac{m_{1}-1}{2}} \tag{R2}
\end{equation*}    
$\bullet$ $m_{2}$ is odd and $m_{1}$ is even:
\begin{align*}
    y_{1}&= x_{m_{2}-1}x_{m_{2}}x_{m_{2}-1}^{-1}\\
    &=  \underbrace{ba\cdots ba}_{m_{2}-1}\,b^{2}\,a\,b^{-2}\, \underbrace{a^{-1}b^{-1}\cdots a^{-1}b^{-1}}_{m_{2}-1}\\
    y_{2} &=  \underbrace{ba\cdots ba}_{m_{2}-1}\,b^{2}\,aba^{-1}\,b^{-2}\,\underbrace{a^{-1}b^{-1}\cdots a^{-1}b^{-1}}_{m_{2}-1}
\end{align*}
{\small \begin{align*}
y_{i}&=
    \begin{cases}
    \begin{multlined}
         \underbrace{ba\cdots ba}_{m_{2}-1}\,b^{2}\,\underbrace{ab\cdots ab}_{i-2}\,aba^{-1} \underbrace{b^{-1}a^{-1}\cdots b^{-1}a^{-1}}_{i-2} \\
       b^{-2}\,\underbrace{a^{-1}b^{-1}\cdots a^{-1}b^{-1}}_{m_{2}-1} 
    \end{multlined}
    & \text{ when}\;i \;\text{is even,}\\
    \begin{multlined}
         \underbrace{ba\cdots ba}_{m_{2}-1}\,b^{2}\,  \underbrace{ab\cdots ab}_{i-1}\,a  \underbrace{b^{-1}a^{-1}\cdots b^{-1}a^{-1}}_{i-1}\\
        b^{-2}\,\underbrace{a^{-1}b^{-1}\cdots a^{-1}b^{-1}}_{m_{2}-1}
    \end{multlined} 
    & \text{ when}\;i \;\text{is odd.}
    \end{cases}
\end{align*}}
\noindent
Then
\begin{align*}
    b &= y_{m_{1}-2}\,y_{m_{1}-3}\,y_{m_{1}-2}^{-1} \\
    &=  \underbrace{ba\cdots ba}_{m_{2}-1}\;b^{2}\, \underbrace{ab\cdots ab}_{m_{1}-2}\;a\;\underbrace{b^{-1}a^{-1}\cdots b^{-1}a^{-1}}_{m_{1}-2}\;b^{-2}\,\underbrace{a^{-1}b^{-1}\cdots a^{-1}b^{-1}}_{m_{2}-1} 
 \end{align*}  
 \begin{align*}
       \implies\; b\, \underbrace{ab\cdots ab}_{m_{2}-1}\,b\, \underbrace{ab\cdots ab}_{m_{1}-2}&= \underbrace{ab\cdots ab}_{m_{2}-1}\,b\, \underbrace{ab\cdots ab}_{m_{1}-2}\,a \\
     b\, \underbrace{ab\cdots ab}_{m_{2}-1}\,b\, \underbrace{ab\cdots ab}_{m_{1}-2}&= \underbrace{ab\cdots ab}_{m_{2}-1}\, \underbrace{ba\cdots ba}_{m_{1}} 
 \end{align*}

\begin{equation*}
b\,(ab)^{\frac{m_{2}-1}{2}}\,b\,(ab)^{\frac{m_{1}-2}{2}}=(ab)^{\frac{m_{2}-1}{2}}\,(ba)^{\frac{m_{1}}{2}} \tag{R3}
\end{equation*}
and
\begin{align*}
    a&= b\,y_{m_{1}-2}\,b^{-1}\\
    & \begin{multlined}
        =b\, \underbrace{ba\cdots ba}_{m_{2}-1}\,b^{2}\, \underbrace{ab\cdots ab}_{m_{1}-4}\,aba^{-1}\, \underbrace{b^{-1}a^{-1}\cdots b^{-1}a^{-1}}_{m_{1}-4}\;b^{-2}\\ 
        \underbrace{a^{-1}b^{-1}\cdots a^{-1}b^{-1}}_{m_{2}-1}\,b^{-1}
    \end{multlined}
\end{align*}
\begin{align*}
     \implies \;  ab\,  \underbrace{ba\cdots ba}_{m_{2}-1}\,b^{2}\, \underbrace{ab\cdots ab}_{m_{1}-4}\,a &= b\, \underbrace{ba\cdots ba}_{m_{2}-1}\,b^{2}\, \underbrace{ab\cdots ab}_{m_{1}-2} \\
     ab\, \underbrace{ba\cdots ba}_{m_{2}-1}\,b\, \underbrace{ba\cdots ba}_{m_{1}-2} &= b\, \underbrace{ba\cdots ba}_{m_{2}-1}\,b^{2}\, \underbrace{ab\cdots ab}_{m_{1}-2}
\end{align*}
\begin{equation*}
    ab\,(ba)^{\frac{m_{2}-1}{2}}\,b\,(ba)^{\frac{m_{1}-2}{2}}=b\,(ba)^{\frac{m_{2}-1}{2}}\,b^{2}\,(ab)^{\frac{m_{1}-2}{2}} \tag{R4}
\end{equation*}
Therefore, the group presentations $G(K_{n,2,m_{1}})$ are given by
\begin{equation}\label{rep1}
    G(K_{n,2,m_{1}}) = \LL a,b \ \bigg\vert
   R1,R2\RR \quad \text{if } \, m_{1}\, \text{is even}
\end{equation}
\begin{equation}\label{rep2}
    G(K_{n,2,2k}) = \LL a,b \ \bigg\vert
    R3,R4\RR \quad \text{if } \,m_{1} \,\text{is odd}.
\end{equation}
\begin{enumerate}
\item[]
Case 1: $m_{1}= (n-1)$ and $m_{2}=n$.
\newline
\vspace{1pt}
\newline
$\bullet$ When $n$ is even, 
\begin{align*}
    G(K_{n,2,n-1}) = \Big<a,b \,\Big\vert\, &(ba)^{\frac{n}{2}}\,(ab)^{\frac{n-2}{2}} = (ab)^{\frac{n-2}{2}}\,a\,(ab)^{\frac{n-2}{2}}\,a ; \\
    & ab\,(ba)^{\frac{n}{2}}\,a\, (ba)^{\frac{n-4}{2}}=b\,(ba)^{\frac{n}{2}}\,(ab)^{\frac{n-2}{2}}\Big> .
\end{align*}
Substituting the first relation in the second relation we have,
\begin{align*}
    ab\,(ba)^{\frac{n}{2}}\,a\, (ba)^{\frac{n-4}{2}} &= b\,(ba)^{\frac{n}{2}}\,(ab)^{\frac{n-2}{2}}\\
    &= b\,(ab)^{\frac{n-2}{2}}\, \,a\,(ab)^{\frac{n-2}{2}}\,a \\
    &= (ba)^{\frac{n}{2}} \,a\,(ba)^{\frac{n-2}{2}}
\end{align*}
\begin{equation}\label{first_eq}
    \implies ab\,(ba)^{\frac{n}{2}} = (ba)^{\frac{n}{2}} \,ab.
\end{equation}
Substituting \eqref{first_eq} in the first relation on the L.H.S repeatedly,
\begin{align*}
    (ba)^{\frac{n}{2}}\,(ab)^{\frac{n-2}{2}}&=(ab)^{\frac{n-2}{2}} \,a\,(ab)^{\frac{n-2}{2}}\,a \\
    (ba)^{\frac{n}{2}}\,ab\,(ab)^{\frac{n-4}{2}} &=(ab)^{\frac{n-2}{2}} \,a\,(ab)^{\frac{n-2}{2}}\,a \\
    \implies ab\,(ba)^{\frac{n}{2}}\,(ab)^{\frac{n-4}{2}} &=(ab)^{\frac{n-2}{2}} \,a\,(ab)^{\frac{n-2}{2}}\,a \\
    \text{Cancelling $ab$ from both sides,}&\\
    \quad (ba)^{\frac{n}{2}}\,(ab)^{\frac{n-4}{2}} &=(ab)^{\frac{n-4}{2}}\,a\,(ab)^{\frac{n-2}{2}}\,a \\
    \vdots & \qquad \qquad \vdots\\
    (ba)^{\frac{n}{2}}\,(ab) &=(ab) \,a\,(ab)^{\frac{n-2}{2}}\,a \\
    ab\,(ba)^{\frac{n}{2}} &= (ab)\,a\,(ab)^{\frac{n-2}{2}}\,a \\
    (ba)^{\frac{n}{2}} &= a^{2}\,(ba)^{\frac{n-2}{2}} \\
    ba &= a^{2} \\
    b&=a.
\end{align*}
This proves that 
\[  G(K_{n,2,n-1})  \cong \mathbb{Z} \quad \text{when $n$ is even}.\]

\noindent
$\bullet$ When $n$ is odd, 
\begin{align*}
    G(K_{n,2,n-1}) = \Big<a,b \,\Big\vert\,  b\,(ab)^{\frac{n-1}{2}}\,b\,(ab)^{\frac{n-3}{2}} &= (ab)^{\frac{n-1}{2}}\,(ba)^{\frac{n-1}{2}} ; \\
     ab\,(ba)^{\frac{n-1}{2}}\,b\, (ba)^{\frac{n-3}{2}} &=b\,(ba)^{\frac{n-1}{2}}\,b^{2}\,(ab)^{\frac{n-3}{2}} \Big> .
\end{align*}
 Now, from the second relation we get
    \begin{align*}
	ab\,(ba)^{\frac{n-1}{2}}\,b\, (ba)^{\frac{n-3}{2}} &= b\,(ba)^{\frac{n-1}{2}}\,b^{2}\,(ab)^{\frac{n-3}{2}} \\
 &= b\cdot b\,(ab)^{\frac{n-1}{2}}\,b\,(ab)^{\frac{n-3}{2}}\\
 &=b \cdot  (ab)^{\frac{n-1}{2}}\,(ba)^{\frac{n-1}{2}} \quad \text{using the first relation}
 \end{align*}
 \begin{equation}\label{second_eq}
     \implies ab\,(ba)^{\frac{n-1}{2}}\,b = (ba)^{\frac{n-1}{2}}\,b^{2}\,a
 \end{equation}
Substituting \eqref{second_eq} in the first relation repeatedly we have,
\begin{align*}
  b\,(ab)^{\frac{n-1}{2}}\,b\,(ab)^{\frac{n-3}{2}} &= (ab)^{\frac{n-1}{2}}\,(ba)^{\frac{n-1}{2}}\\
   (ba)^{\frac{n-1}{2}}\,b^{2}\,a\,b\,(ab)^{\frac{n-5}{2}} &= (ab)^{\frac{n-1}{2}}\,(ba)^{\frac{n-1}{2}}\\
    ab\,(ba)^{\frac{n-1}{2}}\,b\,b\,(ab)^{\frac{n-5}{2}}  &= (ab)^{\frac{n-1}{2}}\,(ba)^{\frac{n-1}{2}} \\
    \text{Cancelling $ab$ from both sides,}&\\
    (ba)^{\frac{n-1}{2}}\,b\,b\,(ab)^{\frac{n-5}{2}}  &= (ab)^{\frac{n-3}{2}}\,(ba)^{\frac{n-1}{2}} \\
    \vdots \qquad & \qquad \vdots \\
     (ba)^{\frac{n-1}{2}} \, b^{2}\,ab &= (ab)^{2}\,(ba)^{\frac{n-1}{2}}\\
     ab\, (ba)^{\frac{n-1}{2}} \, b^{2}&= (ab)^{2}\,(ba)^{\frac{n-1}{2}}
     \end{align*}
\begin{equation}\label{third_eq}
    \implies (ba)^{\frac{n-1}{2}} \, b^{2} = ab\,(ba)^{\frac{n-1}{2}}
\end{equation}
Substituting \eqref{third_eq} in the first relation on the L.H.S repeatedly,
\begin{align*}
    (ba)^{\frac{n}{2}}\,(ab)^{\frac{n-2}{2}}&=(ab)^{\frac{n-2}{2}} \,a\,(ab)^{\frac{n-2}{2}}\,a \\
    (ba)^{\frac{n}{2}}\,ab\,(ab)^{\frac{n-4}{2}} &=(ab)^{\frac{n-2}{2}} \,a\,(ab)^{\frac{n-2}{2}}\,a \\
    \implies ab\,(ba)^{\frac{n}{2}}\,(ab)^{\frac{n-4}{2}} &=(ab)^{\frac{n-2}{2}} \,a\,(ab)^{\frac{n-2}{2}}\,a \\
    \text{cancelling $ab$ from both sides,}&\\
    \quad (ba)^{\frac{n}{2}}\,(ab)^{\frac{n-4}{2}} &=(ab)^{\frac{n-4}{2}}\,a\,(ab)^{\frac{n-2}{2}}\,a \\
    \vdots \qquad & \qquad \vdots \\
    (ba)^{\frac{n}{2}}\,(ab) &=(ab) \,a\,(ab)^{\frac{n-2}{2}}\,a \\
    ab\,(ba)^{\frac{n}{2}} &= (ab)\,a\,(ab)^{\frac{n-2}{2}}\,a \\
    (ba)^{\frac{n}{2}} &= a^{2}\,(ba)^{\frac{n-2}{2}} \\
    ba &= a^{2} \\
    b&=a
\end{align*}
This proves that 
\[  G(K_{n,2,n-1})  \cong \mathbb{Z} \quad \text{when $n$ is even}.\]

\item[] Case 2: $m_{1} \neq (n-1)$.\\
Let, $D_{n}$ denote the Dihedral group of order $2n$.
Now, we will show that all groups $G(K_{n,2,m_{1}}),(1 \leq m_{1} < n-1)$  are not isomorphic to $\Z$ by showing that there exists an epimorphism from these groups to some dihedral group.\\
Let
\begin{align*}
    D_{2(n-m_{1})-1} &= \Big<r,s \, \vert \, r^{2(n-m_{1})-1}=1,s^{2}=1,rs=sr^{-1}\Big> \\
    &\cong \Big< x,y \, \vert \, (xy)^{2(n-m_{1})}x=(yx)^{2(n-m_{1})}y \, , \, x^{2}=y^{2}=1\Big>
\end{align*}
where, $x=s , y= r^{n-m_{1}}s.$\\
Now, we define an epimorphism between the presentations of $G(K_{n,2,m_{1}})$ and $D_{2(n-m_{1})-1}$ by
\[\phi : G(K_{n,2,m_{1}}) \rightarrow  	D_{2(n-m_{1})-1}\]
\[a \rightarrow x\]
\[b \rightarrow y\]
For the relations $R1,R2,R3,R4$ in \eqref{rep1} and \eqref{rep2} we check that
    \[\phi(R1)=\phi(R2)=\phi(R3)=\phi(R4)=(r^{2n-2m_{1}-1})^{n-m_{1}}=1 \]
So, $\phi$ is well defined.    
Let
\[D_{2(n-m_{1})-1} =\Big \{1,r,r^{2},\cdots,r^{2(n-m_{1})-2},s,rs,r^{2}s,\cdots,r^{2(n-m_{1})-2}s\Big \}.\]
\begin{align*}
    \phi^{-1}(r^{n-m_{1}-i})&=(ab)^{2i-1}           &&\text{for } i=1,2,\cdots,n-m_{1}\\
    \phi^{-1}(r^{2(n-m_{1})-(i+1)}) &=(ab)^{2i}     &&\text{for } i=1,2,\cdots,n-m_{1}-1\\
    \phi^{-1}(r^{n-m_{1}-i}s) &=(ab)^{2i-1}a        &&\text{for } i=1,2,\cdots,n-m_{1}\\
    \phi^{-1}( r^{2(n-m_{1})-(i+1)}s) &=(ab)^{2i}a  &&\text{for } i=1,2,\cdots,n-m_{1}-1
\end{align*}
This shows that $\phi$ is surjective.\\
Therefore,
\[ \faktor{G(K_{n,2,m_{1}})}{\ker \phi} \cong D_{2(n-m_{1})-1} \]
which implies
\[ G(K_{n,2,m_{1}}) \ncong \mathbb{Z}.\qedhere\]
\end{enumerate}
\end{proof}
\begin{corollary}~

\begin{enumerate}
    \itemsep=3pt
    \item
    \begin{enumerate}[(a)]
        \itemsep=3pt
        \item $\pi_{1}\big(\R^{4}-\Tube(K_{n,2,m_{1}})\big) \cong \Z$ \, if $m_{1}=n-1$.
        \item $\pi_{1}\big(\R^{4}-\Tube(K_{n,2,m_{1}})\big) 
         \ncong \Z$ \, if $1 \leq m_{1} < n-1$.
    \end{enumerate}
    \item The number of non-trivial ribbon torus knots obtained from the $K(2,2n+1)$ with two weldings does not exceed $n-1$.
\end{enumerate}
   
\end{corollary}
\begin{proof}
    From Proposition 4.2 we can see that the number of welded torus knots $K_{n,2,m_{1}}$ with two weldings have knot group not isomorphic to $\Z$ is $n-1$ for the values of $m_{1}$ from $0$ to $n-2$. Therefore, by Proposition 2.2, the corresponding ribbon torus knots under the tube map will be non-trivial. 
\end{proof}

If we weld more than two crossings in $K(2,2n+1)$, the gaps between any two crossings will matter and it will lead to increase in the computational complexity. 
One can propose to write a computer program to address the most general case.

\subsection{Welded Knots obtained from $T_n$}
\subsubsection{$T_n$ with one welding}
Let $T_{n,1}$ denote the knot diagram obtained after welding one crossing in the twist region of the standard twist knot diagram (Figure~\ref{fig:twist}) having $n$ real crossings in the twist region.
\begin{proposition}
   \begin{align*}
    G(T_{n,1}) & = \Big\langle a,b \mid (ab^{-1})^{\frac{n-2}{2}} a (ba^{-1})^{\frac{n-2}{2}}  = b (ab^{-1})^{\frac{n-4}{2}} a (ba^{-1})^{\frac{n-2}{2}} b^{-1}\,,\\
    & \hspace{4.6cm} b = (ab^{-1})^{\frac{n-4}{2}} a (ba^{-1})^{\frac{n-2}{2}} \Big\rangle \\
    & \cong \Z.
\end{align*} 
\end{proposition}

\subsubsection{$T_n$ with two weldings}
Let $T_{n,2}$ denote the knot diagram obtained after welding two crossings with one classical crossing in between in the twist region of the standard twist knot diagram (Figure~\ref{fig:twist1}).

\begin{figure}[H]
    \centering
    \begin{subfigure}{0.49\textwidth}
        \centering
        \includegraphics{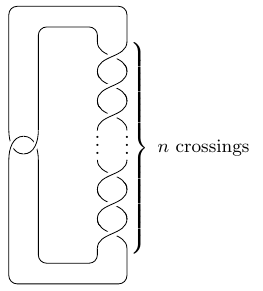}
        \caption{Standard twist knot diagram.}
        \label{fig:twist}
    \end{subfigure}
    \begin{subfigure}{0.49\textwidth}
        \centering
        \includegraphics{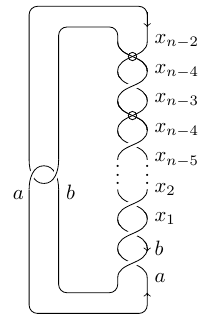}
        \caption{Twist knot with two welds and one gap.}
        \label{fig:twist1}
    \end{subfigure}
    \caption{Twist knot and welded twist knot.}
    \label{}
\end{figure}
\begin{proposition}
Let $G(T_{n,2})$ be the knot group of $T_{n,2}$.
    \[G(T_{n,2})  \ncong \mathbb{Z} \quad \text{when}\, n \, \text{is odd} \]
   \end{proposition}
\begin{proof}
    Let $x_{1},x_{2},\cdots,x_{n-2}$ be the generators for the knot group.
    The Wirtinger relations are the following.
    From the twist region:
\begin{align*}
 	x_{1} &= b^{-1} \,a\,b\\
 	x_{2} &= x_{1} \, b\, x_{1}^{-1}    \\
 	x_{3} &= x_{2} ^{-1}\, x_{1}\, x_{2} \\
 	x_{4} &= x_{3}\,x_{2}\,x_{3}^{-1} . \\
  \end{align*}
Then, for $i \leq (n-3)$
\begin{align*}
    x_{i} =  
    \begin{cases}
       x_{i-1}\,x_{i-2}\,x_{i-1}^{-1}, & \text{if $i$ is even}, \\
       x_{i-1}^{-1}\,x_{i-2}\,x_{i-1}, & \text{if $i$ is odd},
       \end{cases}
\end{align*}
And
\begin{align*}
    x_{n-2} =  
    \begin{cases}
       x_{n-4}^{-1}\,x_{n-3}\,x_{n-4}, & \text{if $n$ is even}, \\
       x_{n-4}\,x_{n-3}\,x_{n-4}^{-1}, & \text{if $n$ is odd}. \\
       \end{cases}
\end{align*}
From the clasp region:\\
 	\[ 
   n \; \text{even}
    \begin{cases}
      x_{n-4}&=a\,x_{n-2}\,a^{-1}\\
           a&= x_{n-4}\,b\,x_{n-4}^{-1}
  \end{cases}
 	\qquad
   n \;\text{odd}
  \begin{cases}
      x_{n-2} &=a^{-1}\,x_{n-4}\,a\\
        a &= x_{n-4}^{-1}\,b\,x_{n-4}
  \end{cases}\]

We can write
\begin{align*}
    	x_{1} &= b^{-1}\,a\,b \\
            x_{2} &= b^{-1}\,a\,b\,a^{-1}\,b \\
            x_{3} &= b^{-1}\,a\,b^{-1}\,a\,b\,a^{-1}\,b \\
            x_{4} &=(b^{-1}\,a)^{2}\,b\,(a^{-1}\,b)^{2} 
\end{align*}
Then, for $i \leq (n-3)$
\begin{align*}
    x_{i} =  
    \begin{cases}
       \underbrace{b^{-1}\,a\cdots b^{-1}\,a}_{i}\,b\,\underbrace{a^{-1}\,b\cdots a^{-1}\,b}_{i} \quad \text{if} \quad i \;\text{is even} \\
       \underbrace{b^{-1}\,a\cdots b^{-1}\,a }_{i-1}\,b^{-1}\,a\,b\,\underbrace{a^{-1}\,b\cdots a^{-1}\,b}_{i-1} \quad \text{if} \quad i \;\text{is odd} \\
       \end{cases}
\end{align*}
\begin{align*}
    x_{n-2} =  
    \begin{cases}
       \underbrace{b^{-1}\,a\cdots b^{-1}\,a }_{n-4}\,b^{-2}\,a\,b^{2}\,\underbrace{a^{-1}\,b\cdots a^{-1}\,b}_{n-4} \quad \text{if} \quad n \; \text{is even} \\
       \underbrace{b^{-1}\,a\cdots b^{-1}\,a }_{n-3}\,a\,b\,a^{-1}\,\underbrace{a^{-1}\,b\cdots a^{-1}\,b}_{n-3} \quad \text{if} \quad n \; \text{is odd}\\
       \end{cases}
\end{align*}
Now, clasp region gives the relations for the presentation:\\
\text{When $n\,(n>3)$ is even},
\begin{align*}
     G(K_{n,2}) &= \Big<a,b \,\Big\vert\,R_{5}, R_{6} \Big> \\
     \text{where}\\
     R_{5} &: (b^{-1}\,a)^{\frac{n-4}{2}}\,b\,(a^{-1}\,b)^{\frac{n-4}{2}}\,a = a\,(b^{-1}\,a)^{\frac{n-4}{2}}\,b^{-2}\,a\,b^{2}\,(a\,b^{-1})^{\frac{n-4}{2}}\,a  \\
     R_{6} &:(a\,b^{-1})^{\frac{n-4}{2}}\,a\,(b\,a^{-1})^{\frac{n-4}{2}} =(b^{-1}\,a)^{\frac{n-4}{2}}\,b\,(a^{-1}\,b)^{\frac{n-4}{2}} 
\end{align*}
When $n\,(n>3)$ is odd,
\begin{align*}
     G(K_{n,2}) &= \Big<a,b \,\Big\vert\,R_{7}, R_{8} \Big> \\
     \text{where}\\
     R_{7} &: a\,(b^{-1}\,a)^{\frac{n-3}{2}}\,a\,b\,a^{-1}\,(a^{-1}\,b)^{\frac{n-3}{2}} = (b\,a^{-1})^{\frac{n-5}{2}}\,b^{-1}\,a\,b\,(a\,b^{-1})^{\frac{n-5}{2}}\,a  \\
     R_{8} &:(b^{-1}\,a)^{\frac{n-5}{2}}\,b^{-1}\,a\,b\,(a^{-1}\,b)^{\frac{n-5}{2}} \,a= b\,(b^{-1}\,a)^{\frac{n-5}{2}}\,b^{-1}\,a\,b\,(a^{-1}\,b)^{\frac{n-5}{2}}
\end{align*}
Let 
\begin{align*}
    D_{2n-7} &= \LL r,s \, \vert \, r^{2n-7}=1,s^{2}=1, rs=sr^{-1} \RR\\
    &\cong \LL x,y \,\vert\, (xy)^{2(n-3)}x=(yx)^{2(n-3)}y \, , \, x^{2}=y^{2}=1 \RR.
\end{align*}
where $x=s, y= r^{n-3}s$.\\
Now, we define an epimorphism between the presentations of $G(T_{n,2})$ $(n>3,odd)$ and $D_{2n-7}$ by
\[\psi : G(K_{n,2}) \rightarrow  	D_{2n-7}\]
\[a \rightarrow x\]
\[b \rightarrow y\]
Therefore,
\[ \faktor{G(T_{n,2})}{\ker \psi} \cong D_{2n-7} \]
which implies
\[ G(T_{n,2}) \ncong \mathbb{Z}.\]
\end{proof}
\begin{corollary}
    If $n$ is odd, then $\pi_1(\R^4-\Tube(T_{n,2})) \ncong \Z$.
\end{corollary}
\begin{proof}
    The proof follows directly from Theorem~\ref{thm:groupsTK}.
\end{proof}
\begin{conjecture}
    If $n$ is even, then $G(T_{n,2})\ncong \mathbb{Z}$.
\end{conjecture}

\section{Knots from Rolfsen's table upto 6 Crossings}	
\noindent
In the previous sections we discover some families by welding at most two crossings because the number of cases will increase rapidly and that will become computationally difficult. In this section we consider standard diagrams of classical knots with crossing number at most six so that we can observe every case possible. Now, any non-trivial welded knot $K$ satisfies $c(K) \geq 3$, where $c(K)$ is the number of classical crossings \cite[Lemma 2.3]{Satoh1}.
Therefore, at most one and two crossings can be welded in the knots with four crossings ($4_{1}$) and with five crossings $(5_{1},5_{2})$ respectively.
These cases are already discussed in the previous sections.
The following tables include the fundamental groups of welded knots obtained by welding at most three crossings of the standard diagrams of the six crossing knots.
It is clear that welding more than three crossings will lead to the welded unknot~\cite[Lemma 2.3]{Satoh1}. In Figure~\ref{fig:6crossings}, the crossings are denoted by 
$c_{i},i=1,\cdots,6$.
The first column of each table indicates which crossings have been welded.
\begin{figure}[ht]
    \centering
    \includegraphics[scale=0.75]{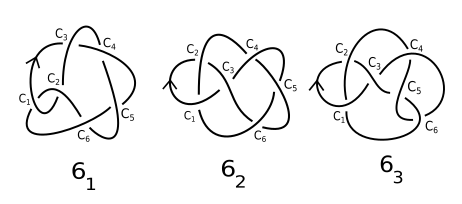}
    \caption{Knots with six crossings.}
    \label{fig:6crossings}
\end{figure}
\newpage
$\bullet$ {\bf $6$ crossing knots with one weld:}\\

\begin{TAB}(r,0.3cm,0.5cm)[5pt]{|c|c|}{|c|c|}
\centering
\shortstack{$\mathbf{6_{1},6_{2},6_{3}}$\\  {\bf one weld}} &\shortstack{{\bf Welded knot group} \\ $\mathbf{G(K)}$}\\
    
    $\{c_{i},\; \forall i=1,\cdots,6\}$ & $\cong \Z$
\end{TAB}
\newline
\vspace*{2mm}
\newline
$\bullet$ $\mathbf{6_{1}}$ {\bf with two weldings:}\\

\begin{TAB}(r,0.3cm,0.5cm)[5pt]{|c|c|}{|c|c|c|c|c|}
\centering
\shortstack{{\bf pair of welded crossings} \\ $\mathbf{\{c_{i},c_{j}\}} \; i<j$}   &\shortstack{{\bf Welded knot group} \\ $\mathbf{G(K)}$}\\
     
 \shortstack{$\{c_{1},c_{j}\}, \forall j=3,4,5,6$, $\{c_{2},c_{j}\}, \forall j=3,4,5,6$ \\$\{c_{3},c_{5}\}$,$\{c_{4},c_{6}\}, \{c_{5},c_{6}\} $  }& $\cong \Z$    \\ 
     
$\{c_{1},c_{2}\}$ & \shortstack{ $ \cong \Z$ \\ welded unknot}  \\
     
$\{c_{i},c_{i+1}\}, \forall j=3,4,5$ & $\cong \Big<a,b \vert a^{-1}bab^{-1}aba^{-1}b^{-1}ab^{-1}=e \Big> \ncong \Z$   \\
      
$\{c_{3},c_{6}\}$ &  $\cong \Big<a,b \vert b^{-1}ab=ab^{-1}a \Big> \ncong \Z$ 
\end{TAB}
\newline
\vspace*{2mm}
\newline
$\bullet$ $\mathbf{6_{1}}$ {\bf with three weldings:}\\

\begin{TAB}(r,0.3cm,0.5cm)[5pt]{|c|c|}{|c|c|}
\centering
 $\mathbf{\{c_{i},c_{j},c_{k}\},\; i<j<k}$   & \shortstack{{\bf Welded knot group} \\ $\mathbf{G(K)}$}\\
 $\{c_{i},c_{j},c_{k}\} \forall \,i,j,k=1,\cdots,6$ & $\cong \Z$ 
\end{TAB}
\newline
\vspace*{2mm}
\newline
\newpage
$\bullet$ $\mathbf{6_{2}}$ {\bf with two weldings:}\\

\begin{TAB}(r,0.3cm,0.5cm)[5pt]{|c|c|}{|c|c|c|c|c|c|c|}
\centering
 \shortstack{{\bf pair of welded crossings} \\ $\mathbf{\{c_{i},c_{j}\},\; i<j}$}   & \shortstack{{\bf Welded knot group} \\ $\mathbf{G(K)}$}\\
 
\shortstack{$\{c_{1},c_{j}\}, \forall j=3,4,5$,\\ $\{c_{2},c_{j}\}, \forall j=3,5,6$ } & $\cong \Z$  \\

\shortstack{$\{c_{3},c_{j}\}, \forall j=4,5,6$ \\
$\{c_{4},c_{6}\}$   } & $\cong \Z$  \\

$\{c_{4},c_{2}\}$ &  $\cong \Big<a,b \vert a^{2}=b^{2},b^{-1}ab=ab^{-1}a \Big> \ncong \Z$  \\

$\{c_{4},c_{5}\},\{c_{5},c_{6}\}$  & $\cong \Big<a,b \vert a^{-1}ba^{-1}b^{-1}aba^{-1}bab^{-1} \Big> \ncong \Z$    \\

$\{c_{2},c_{4}\}$  & $\cong \Big<a,b \vert b^{2}=a^{2},a^{-1}ba=ba^{-1}b \Big> \ncong \Z$    \\

$\{c_{5},c_{6}\}$  & $\cong \Big<a,b \vert aba=bab \Big> \ncong \Z$   
\end{TAB}
\newline
\vspace*{2mm}
\newline
$\bullet$ $\mathbf{6_{2}}$ {\bf with three weldings:}\\

\begin{TAB}(r,0.3cm,0.5cm)[5pt]{|c|c|}{|c|c|c|c|c|c|}
\centering
 $\mathbf{\{c_{i},c_{j},c_{k}\},\; i<j<k}$   & \shortstack{{\bf Welded knot group} \\ $\mathbf{G(K)}$}\\

\shortstack{ $\{c_{1},c_{3},c_{4}\},\{c_{1},c_{3},c_{5}\},\{c_{1},c_{3},c_{6}\},$ \\ $\{c_{1},c_{4},c_{5}\}, \{c_{1},c_{4},c_{6}\},\{c_{1},c_{5},c_{6}\}$ } & $\cong \Z$ \\
  
 \shortstack{$\{c_{1},c_{2},c_{4}\},\{c_{2},c_{3},c_{5}\},\{c_{2},c_{3},c_{6}\},\{c_{2},c_{4},c_{5}\}$ \\ $\{c_{2},c_{5},c_{6}\},\{c_{3},c_{4},c_{5}\},\{c_{3},c_{4},c_{6}\},\{c_{4},c_{5},c_{6}\}$ } & $\cong \Z$ \\
 
$\{c_{1},c_{2},c_{3}\},\{c_{1},c_{2},c_{5}\},\{c_{1},c_{2},c_{6}\}$  & $\cong \Big<a,b \vert aba=bab \Big> \ncong \Z$ \\

$\{c_{2},c_{3},c_{4}\}$  & $\cong \Big<a,b \vert a^{2}=b^{2},a^{-1}ba=ba^{-1}b \Big> \ncong \Z$    \\

$\{c_{2},c_{4},c_{6}\}$  & $\cong \Big<a,b \vert bab^{-1}=aba^{-1} \Big> \ncong \Z$    
\end{TAB}
\newline
\vspace*{2mm}
\newline
\newpage
$\bullet$ $\mathbf{6_{3}}$ {\bf with two weldings:}\\

\begin{TAB}(r,0.3cm,0.5cm)[5pt]{|c|c|}{|c|c|c|c|c|c|}
\centering
 \shortstack{{\bf pair of welded crossings} \\ $\mathbf{\{c_{i},c_{j}\},\; i<j}$}   & \shortstack{{\bf Welded knot group} \\ $\mathbf{G(K)}$}\\
\shortstack{$\{c_{1},c_{j}\}, \forall j=3,4,5,6$,\\ $\{c_{2},c_{j}\}, \forall j=3,5,6$ } & $\cong \Z$  \\

\shortstack{$\{c_{3},c_{j}\}, \forall j=4,5,6$ \\
$\{c_{4},c_{j}\}, \forall j=5,6$   } & $\cong \Z$  \\

$\{c_{1},c_{2}\}$  & $\cong \Big<a,b \vert aba^{-1}=b^{-1}ab \Big> \ncong \Z$    \\

$\{c_{2},c_{4}\}$  & $\cong \Big<a,b \vert b^{2}=a^{2},a^{-1}ba=ba^{-1}b \Big> \ncong \Z$    \\

$\{c_{5},c_{6}\}$  & $\cong \Big<a,b \vert aba=bab \Big> \ncong \Z$   
\end{TAB}
\newline
\vspace*{2mm}
\newline
$\bullet$ $\mathbf{6_{3}}$ {\bf with three weldings:}\\

\begin{TAB}(r,0.3cm,0.5cm)[5pt]{|c|c|}{|c|c|c|c|c|}
\centering
$\mathbf{\{c_{i},c_{j},c_{k}\},\; i<j<k}$   & \shortstack{{\bf Welded knot group} \\ $\mathbf{G(K)}$}\\
 
 $\{c_{1},c_{2},c_{6}\},\{c_{2},c_{3},c_{5}\},\{c_{2},c_{3},c_{6}\},\{c_{3},c_{4},c_{6}\}$  & $\cong \Z$ \\
 
$\{c_{1},c_{2},c_{3}\},\{c_{1},c_{2},c_{4}\},\{c_{1},c_{2},c_{5}\},\{c_{4},c_{5},c_{6}\}$  & $\cong \Big<a,b \vert aba=bab \Big> \ncong \Z$ \\

$\{c_{2},c_{3},c_{4}\}$  & $\cong \Big<a,b \vert a^{2}=b^{2},ab^{-1}a=b^{-1}ab \Big> \ncong \Z$    \\

$\{c_{3},c_{4},c_{5}\}$  & $\cong \Big<a,b \vert a^{2}=b^{2},ab^{-1}a=ba^{-1}b \Big> \ncong \Z$    
\end{TAB}

\section{Conclusion and Future Directions}
\noindent 
In this paper, we generated infinite families of non-trivial ribbon torus-knots using the tube map by welding at most two classical crossings of the standard diagrams of torus and twist knots.
It is observed that the non-triviality of these knots depends not only on the positions of the welding but also the number of crossings between them.
Some non-trivial ribbon torus knots with knot group not isomorphic to $\Z$ are given in the preceding sections, associated to the welded torus-knot family $K(2,2n+1)$ and the welded twist knot family $T_{n}$. 
We are unable to decide the triviality of welded knots with knot group $\Z$ since the unknotting conjecture has not been proved yet.
In certain cases, we employed the welded unknotting number to confirm whether or not the welded knot belonging to knot group $\Z$ is a welded unknot.
As a future work, we would like to study bounds of the welded unknotting number and generate more examples of non-trivial ribbon torus-knots by considering other families of classical knots and welding more than two crossings.
	
\section*{Acknowledgment}	The authors are thankful for the valuable discussions with Louis Kauffman.

\end{document}